\documentclass[10pt,electronic]{amsart}        
\usepackage[applemac]{inputenc}
\usepackage[T1]{fontenc}
\usepackage[small,it]{caption}
\usepackage[english]{babel}
\usepackage{url}
\usepackage[hmargin=3.7cm,vmargin=3.4cm]{geometry} 
\usepackage{graphicx}                         
\usepackage{amsmath}
\usepackage{amsthm}                      
\usepackage{amssymb}
\usepackage{amscd}     
\usepackage{mathrsfs}
\usepackage{stmaryrd}
\usepackage{enumerate}
\usepackage{euscript}
\renewcommand{\mathcal}{\EuScript}

\usepackage[isbn=false,doi=false,eprint=false,url=false,style=authoryear,dashed=false,backend=bibtex]{biblatex}
\makeatletter
\newrobustcmd*{\parentexttrack}[1]{%
  \begingroup
  \blx@blxinit
  \blx@setsfcodes
  \blx@bibopenparen#1\blx@bibcloseparen
  \endgroup}
\AtEveryCite{%
  }
\makeatother
\renewcommand{\cite}{\parencite}
\DeclareNameAlias{sortname}{last-first}
\renewbibmacro*{volume+number+eid}{%
  \printfield{volume}%
  \setunit*{\addnbspace}
  \printfield{number}%
  \setunit{\addcomma\space}%
  \printfield{eid}}
 \DeclareFieldFormat[article]{volume}{\textbf{#1}}
\DeclareFieldFormat[article]{number}{{\mkbibparens{#1}},}
\DeclareFieldFormat{pages}{#1 pages}
\DeclareFieldFormat
  [article,inbook,incollection,inproceedings,patent,thesis,unpublished]
  {pages}{#1}
\renewbibmacro{in:}{}
\theoremstyle{plain}                                                           
\newtheorem{thm}{Theorem}[section]

\newtheorem{prop}[thm]{Proposition}

\theoremstyle{definition}

\newtheorem{rem}[thm]{Remark}

\DeclareMathOperator{\codim}{codim}
\newcommand{\field}[1]{\ensuremath{\mathbf{#1}}}
\newcommand{\Q}{\ensuremath{\field{Q}}}        
\newcommand{\C}{\ensuremath{\mathcal{C}}}

\newcommand{\Z}{\ensuremath{\field{Z}}} 

\newcommand{\p}{\ensuremath{\field{P}}}

\newcommand{\M}{\mathcal{M}}
\newcommand{\MM}{\overline{\mathcal{M}}}

\newcommand{\Bl}{\mathrm{Bl}}
\newcommand{\st}{\,:\,}

\title{Poincar\'e duality of wonderful compactifications and tautological rings}

\author{Dan Petersen}
\thanks{The author is supported by the Danish National Research Foundation through the
Centre for Symmetry and Deformation (DNRF92).}
\email{danpete@math.ku.dk}
\address{Institut for Matematiske Fag \\
K{\o}benhavns Universitet \\
Universitetsparken 5 \\
2100 K{\o}benhavn {\O} }
\begin{document} 
 \maketitle   
 
 \begin{abstract}Let $g \geq 2$. Let $\M_{g,n}^{rt}$ be the moduli space of $n$-pointed genus $g$ curves with rational tails. Let $\C_g^n$ be the $n$-fold fibered power of the universal curve over $\M_g$. We prove that the tautological ring of $\M_{g,n}^{rt}$ has Poincar\'e duality if and only if the same holds for the tautological ring of $\C_g^n$. We also obtain a presentation of the tautological ring of $\M_{g,n}^{rt}$ as an algebra over the tautological ring of $\C_g^n$. This proves a conjecture of Tavakol. Our results are valid in the more general setting of wonderful compactifications. 
 \end{abstract}
 
 \section{introduction}

Let $g \geq 2$, and let $\M_g$ be the moduli space of smooth curves of genus $g$. Let $\C_g \to \M_g$ be the universal curve, and let $\C_g^n$ be its $n$-fold fibered power. We denote by $R^\bullet(\C_g^n)$ the \emph{tautological ring} of $\C_g^n$, which is a subalgebra of its rational Chow ring $A^\bullet(\C_g^n)$ generated by certain geometrically natural classes. The tautological ring was introduced by Mumford \cite{mumfordtowards} and has been intensely studied since then, in particular because of the series of conjectures known as the \emph{Faber conjectures} \cite{faberconjectures,pandharipandequestions} and because of the role it plays in Gromov--Witten theory. By \cite{looijengatautological,fabernonvanishing}, it is known that $R^{g-2+n}(\C_g^n) \cong \Q$ and $R^k(\C_g^n) = 0$ for $k > g-2+n$. One part of the {Faber conjectures} asserts that $R^\bullet(\C_g^n)$ is a Poincar\'e duality algebra with socle in degree $g-2+n$; that is, that the cup-product pairing into the top degree is perfect. 



One can also consider the space of $n$-pointed curves of genus $g$ with \emph{rational tails}, $\M_{g,n}^{rt}$. The space $\M_{g,n}^{rt}$ is defined as the preimage of $\M_g$ under the forgetful map $\MM_{g,n} \to \MM_g$ between Deligne--Mumford compactifications. It, too, has a tautological ring, and 
according to the Faber conjectures $R^\bullet(\M_{g,n}^{rt})$ is also a Poincar\'e duality algebra with socle in degree $g-2+n$. It is believed among experts that these two conjectures `should' be equivalent. For example, \cite[Appendix A]{pixtonthesis} writes: 

\begin{quote}
``Also, instead of doing computations in $\M^{rt}_{g,n}$ we work with $\C_g^n$, the $n$th power of the universal curve over $\M_g$. The tautological rings of these two spaces are very closely related, and it seems likely that the Gorenstein discrepancies are always equal in these two cases.''
\end{quote}

However, I am not aware of any precise result along these lines in the literature. The goal of this note is to prove an explicit relationship between the two tautological rings, which in particular implies that $R^\bullet(\M_{g,n}^{rt})$ will have Poincar\'e duality if and only the same is true for $R^\bullet(\C_g^n)$ (Theorem \ref{maincor1}). In fact, our results give an expression for the Gorenstein discrepancies of $R^\bullet(\M_{g,n}^{rt})$ in terms of those of $R^\bullet(\C_g^m)$  for $1 \leq m \leq n$. We also deduce a presentation of $R^\bullet(\M_{g,n}^{rt})$ as an algebra over $R^\bullet(\C_g^n)$ (Proposition \ref{fmcor}) from known results on the Chow rings of Fulton--MacPherson compactifications. This proves a  conjecture of Tavakol. 

We remark that it is likely that $R^\bullet(\M_{g,n}^{rt})$ does \emph{not} have Poincar\'e duality in general. Counterexamples to the analogous conjectures for the spaces $\MM_{g,n}$ and $\M_{g,n}^{ct}$ have been constructed in \cite{petersentommasi,m28ct}. The conjecture that Pixton's extension of the Faber--Zagier relations give rise to all relations in the tautological rings would imply that $R^\bullet(\M_{g,n}^{rt})$ fails to have Poincar\'e duality in general \cite{pixtonthesis,pandharipandepixtonzvonkine,jandachow}; as would Yin's conjecture that all relations on the symmetric power $\C_g^{[n]}$ should arise from motivic relations on the universal jacobian \cite{yin}. See also the discussion in \cite{pcmifaber}.

What we prove is a  general result about intersection rings of \emph{wonderful compactifications} \cite{wonderfulcompactification}. Let $Y$ be a smooth variety, and $G$ a collection of subvarieties of $Y$ which form a \emph{building set} (see Subsection \ref{wonderfulsection}). The wonderful compactification $Y_G$ is obtained from $Y$ by a sequence of blow-ups in smooth centers, given by intersections of the elements of $G$. Our motivation for considering these is that if we take $Y = \C_g^n$ and $G$ the collection of all diagonal loci, then $Y_G \cong \M_{g,n}^{rt}$. It is perhaps worth pointing out that in this particular case --- where $Y$ is given by an $n$-fold cartesian power, and $G$ consists of all diagonals --- the wonderful compactification reduces to the compactification introduced in \cite{fmcompactification}. 

By a result of \cite{limotives}, the Chow ring of $Y_G$ can be expressed  in a combinatorial fashion in terms of the Chow rings of $Y$ and the Chow rings of certain intersections of elements of $G$, which we call \emph{burrows}, see Theorem \ref{lithm}. An identical formula works also for the respective cohomology rings. Suppose now that $Y$ is also compact. Then the cohomology rings of $Y$ and all burrows will have Poincar\'e duality, and so will the cohomology ring of $Y_G$ (as all these spaces are smooth and compact). One might therefore guess that Poincar\'e duality of $H^\bullet(Y_G)$ can be deduced purely combinatorially from Poincar\'e duality for $Y$ and for the burrows. This turns out to be true, and quite easy to prove (independently of Li's result mentioned in the first sentence of this paragraph): the inductive structure of the wonderful compactification implies that one only needs to check that Poincar\'e duality is preserved under two `basic' operations, where we either blow up in a single smooth center or form a projective bundle. In particular, the same phenomenon --- that Poincar\'e duality for $Y$ and for all burrows implies Poincar\'e duality of $Y_G$ --- will work equally well on the level of Chow rings, and for not necessarily compact $Y$. Moreover, as will be crucial for us, the argument works identically for the tautological rings.

\section{Poincar\'e duality and wonderful compactifications}

\subsection{Wonderful compactifications}\label{wonderfulsection}
We refer to the papers \cite{wonderfulcompactification,limotives} for precise definitions of \emph{wonderful compactification}, \emph{building set} and \emph{nest}, as well as for proofs of the below assertions. Instead we state only as much as is needed for the logic of the proof.

Let $Y$ be a smooth variety. Let $G$ be a \emph{building set} in $Y$. This means that $G$ is a collection of closed subvarieties of $Y$ satisfying certain conditions regarding the combinatorics of how the varieties in $G$ may intersect each other. These conditions state in particular that any nonempty intersection of elements of $G$ is smooth. We call such intersections \emph{burrows}. This terminology does not appear in Li's work; it is motivated by the fact that a burrow is the intersection of the elements of a \emph{nest}. In particular, $Y$ itself is a burrow, corresponding to the empty intersection.

The \emph{wonderful compactification} $Y_G$ is obtained from $Y$ by an iterative procedure as follows: let $X$ be an  element of minimal dimension\footnote{As explained by Li, there are other possible orders in which one can perform the blow-ups, but this one will suffice for our purposes.} in $G$, and let $Y^{(1)}= \Bl_X Y$. There is an induced building set $G^{(1)}$ in $ Y^{(1)}$ which consists of the strict transforms of all elements of $G$. 
We define (by induction) $Y_G =  Y^{(1)}_{G^{(1)}}$. The reason this makes sense as an inductive definition is that eventually we obtain a variety $Y^{(n)}$ with a building set $G^{(n)}$ all of whose elements are Cartier divisors, after which all further blow-ups are canonically isomorphisms and $Y_G = Y^{(n)}$. Every burrow for the building set $G^{(1)}$ is obtained from a burrow for $G$ either as a blow-up in a smaller burrow, or as a projective bundle of some rank. More precisely, if $\{Z_\alpha\}_{\alpha \in A}$ is the set of burrows for $G$, then the set of burrows for $G^{(1)}$ is given by $$\{\widetilde Z_\alpha\}_{\alpha \in A} \cup \{E_X \cap \widetilde Z_\alpha\}_{\substack{\alpha \in A \\ \emptyset \neq Z_\alpha \cap X \neq X }},$$ where $E_X$ denotes the exceptional divisor and $\sim$ denotes the dominant transform. Recall that the dominant transform coincides with the strict transform except for varieties contained in the center of the blow-up, in which case the strict transform is empty and the dominant transform is the preimage of the subvariety. 

The wonderful compactification $Y_G$ is again smooth, containing $Y^\circ = Y \setminus \bigcup_{X \in G} X$ as the complement of a strict normal crossing divisor. If $Y$ is compact, then $Y_G$ is a compactification of $Y^\circ$, which explains the awkward terminology (specifically, that a `wonderful compactification' is not necessarily compact). In practice one often starts with the space $Y^\circ$ and one wishes to compactify it so that its complement is a normal crossing divisor, which is useful e.g.\ in mixed Hodge theory \cite{hodge2}. By Hironaka's theorem such a compactification always exists, but the advantage of $Y_G$ is its explicit description and combinatorial structure. The irreducible components of the normal crossing divisor correspond bijectively to the elements of $G$. 


\begin{rem} Here are some examples of wonderful compactifications. 
\begin{enumerate}
\item Let $X$ be a smooth variety, and $Y=X^n$ its $n$-fold cartesian power. For every subset $I \subseteq [n]$ we have a diagonal $$D_I = \{(x_1,\ldots,x_n) \in Y \st x_i = x_j \text{ for } i, j \in I\}.$$ The collection of all diagonals $D_I$ with $\vert I \vert \geq 2$ form a building set, and the corresponding wonderful compactification is the Fulton--MacPherson compactification \cite{fmcompactification} of $Y^\circ = F(X,n)$, the configuration space of $n$ distinct ordered points on $X$.
\item As remarked in Fulton and MacPherson's original paper, their construction makes sense just as well for a smooth family of algebraic varieties $X \to S$ over a smooth base, and the collection of diagonals in the $n$-fold fibered power of $X$ over $S$. In particular, we can consider the universal family $\C_g \to \M_g$ over the moduli space of curves of genus $g$ and its $n$-fold fibered power $\C_g^n$. For $I \subseteq [n]$ we have the diagonals $D_I = \{(C;x_1,\ldots,x_n) \in \C_g^n \st x_i = x_j \text{ for } i, j \in I\}$ which form a building set (this time we can even take $\vert I \vert \geq 3$ if we wish, since the diagonals with $\vert I \vert = 2$ are already Cartier divisors). The resulting wonderful compactification is exactly $\M_{g,n}^{rt}$.
\item In the previous examples, we can instead consider \emph{polydiagonals}, i.e.\ arbitrary intersections of diagonals. These also form a building set, and the corresponding wonderful compactification was first introduced by \cite{ulyanov}.
\item Let $Y$ be a complex vector space, and suppose that the building set $G$ is a collection of subspaces. In this case, $Y_G$ is the \emph{wonderful model} of the subspace arrangement, introduced by \cite{deconciniprocesiwonderful}. 
\item Let $Y = \p^n$, and choose $n+2$ points in general position on $Y$. Let $G$ be the collection of all projective subspaces spanned by these points. This is a building set, and $Y_G \cong \MM_{0,n+3}$, by a construction of \cite{kapranovchow}.
\item Let $Y = (\p^1)^n$, and take $G$ to be the set of all diagonals $D_I$, as well as all subsets of the form
$$ D_{I,p} = \{(x_1,\ldots,x_n) \in Y \st x_i = p \text{ for } i \in I\}$$
for $p \in \{0,1,\infty\}$. In this example, too, $Y_G \cong \MM_{0,n+3}$ \cite{keel}. 
\end{enumerate}
\end{rem}
 
\subsection{Preservation of Poincar\'e duality} 
 
In this section we shall consider operations on algebraic varieties which preserve the property of having a Chow ring with Poincar\'e duality. This property may seem a bit unnatural, except in the very special case of a smooth compact variety whose Chow ring maps isomorphically onto the cohomology ring (e.g.\ one with an algebraic cell decomposition).  Nevertheless we can of course study it. Later we will observe that all propositions below remain valid if the Chow rings are replaced with tautological rings, in the cases we are interested in. 
 

 \begin{prop} \label{prop1} Let $Y$ be a smooth variety and $i \colon Z \hookrightarrow Y$ a smooth closed subvariety of codimension $c$. Suppose that $A^d(Y) \cong A^{d-c}(Z) \cong \Q$ for some integer $d$, and that both intersection rings vanish above these degrees, and that $0 \neq [Z] \in A^\bullet(Y)$. The following are equivalent:
 \begin{enumerate}[\qquad\scshape(1)]
 \item $A^\bullet(\Bl_Z Y)$ has Poincar\'e duality.
 \item $A^\bullet(Y)$ and $A^\bullet(Z)$ both have Poincar\'e duality. 
\end{enumerate}  \end{prop}

\begin{proof}Under the hypotheses, we have
\begin{equation} \tag{*}
A^i(\Bl_Z Y) \cong A^i(Y) \oplus \bigoplus_{k=1}^{c-1}  A^{i-k}(Z) \cdot E^k, \label{chowringdecomposition}
\end{equation}
where $E$ stands for the class of the exceptional divisor. The ring structure is given by the rules
$$ \alpha \cdot E = i^\ast(\alpha) \cdot E$$
for $\alpha \in A^\bullet(Y)$, and 
$$ \alpha \cdot E^c = (-1)^ci_\ast(\alpha) + \sum_{i=1}^{c-1} (-1)^i\alpha c_i \cdot E^{c-i}$$
for $\alpha \in A^\bullet(Z)$, where $c_i \in A^i(Z)$ is the $i$th Chern class of the normal bundle $N_{Z\subset Y}$. It follows easily that the intersection matrix describing the pairing $$A^i(\Bl_Z Y) \otimes A^{d-i}(\Bl_Z Y) \to A^d(\Bl_Z Y) \cong \Q$$
becomes block upper triangular when the summands in \eqref{chowringdecomposition} are ordered as $$1, E, E^2, \ldots, E^{c-1}$$ in degree $i$, and ordered as $$1, E^{c-1}, E^{c-2}, \ldots, E $$
in degree $d-i$. Explicitly, the matrix will take the form
$$ \begin{bmatrix}
\ast & 0 & 0 & 0 & \cdots & 0 \\
0 & \ast & \ast & \ast & \cdots & \ast \\
0 & 0 & \ast & \ast & \cdots & \ast \\
0 & 0 & 0 & \ast & \cdots & \ast \\
\vdots & \vdots & \vdots & \vdots & \ddots & \vdots \\
0 & 0 & 0 & 0 & \cdots & \ast   
\end{bmatrix},$$
where each $\ast$ denotes a block of the matrix. Now the first diagonal block is given by the intersection pairing $A^i(Y) \otimes A^{d-i}(Y) \to A^d(Y) \cong \Q$, and the remaining are given (up to a scalar) by the pairings $A^{i-k}(Z) \otimes A^{d-i-c+k}(Z) \to A^{d-c}(Z) \cong \Q$ for $k=1,\ldots,c-1$. 
Thus if $A^\bullet(\Bl_Z Y)$ has Poincar\'e duality, i.e.\ if this matrix is invertible for all $i$, then $A^\bullet(Y)$ and $A^\bullet(Z)$ must also have Poincar\'e duality. 

 Assume conversely that $A^\bullet(Y)$ and $A^\bullet(Z)$ have Poincar\'e duality. This is nearly enough to conclude that the matrix above is invertible, except if $i_\ast \colon A^{d-c}(Z) \to A^d(Y)$ vanishes, in which case all but the first diagonal block will be zero. But since we supposed in addition that $[Z] \neq 0$, then since $A^\bullet(Y)$ has Poincar\'e duality, there is a class $\alpha \in A^{d-c}(Y)$ with $\alpha [Z] \neq 0$. Then $i_\ast i^\ast \alpha \neq 0$, so we are done. 
\end{proof}

 \begin{prop} \label{prop2} Let $Y$ be a smooth variety and let $E \to Y$ be a rank $r$ vector bundle. Then $A^\bullet(Y)$ has Poincar\'e duality (with socle in degree $d$) if and only if $A^\bullet(\mathbf P(E))$ has Poincar\'e duality (with socle in degree $d+r-1$).
 \end{prop}

\begin{proof}The proof is analogous to the previous one, but simpler, using the projective bundle formula. 
\end{proof}

An analogous statement can be proven for a {wonderful compactification} by iterating the previous two propositions. 

\begin{prop} \label{poincaredualitywonderful}Let $Y$ be a smooth variety. Let $G$ be a building set in $Y$. Suppose that there exists an integer $d$ such that for every  burrow $Z \subseteq Y$ (including $Z=Y$), $A^{d -\codim Z}(Z) \cong \Q$ and $A^\bullet(Z)$ vanishes above this degree, and $[Z] \neq 0$ in $A^\bullet(Y)$. Assume also that all restriction maps $A^\bullet(Y) \to A^\bullet(Z)$ are surjective (equivalently, the restriction map for any pair of burrows is surjective). Then the following are equivalent: 
\begin{enumerate}[\qquad\scshape(1)]
\item $A^\bullet(Y_G)$ has Poincar\'e duality.
\item For every burrow $Z$, $A^\bullet(Z)$ has Poincar\'e duality.
\end{enumerate}
\end{prop}

\begin{proof}	Let $Y^{(1)}$ be the variety obtained by blowing up an element  $X \in G$ of minimal dimension, as described in Section \ref{wonderfulsection}. 
	
	We first argue that if $Z^{(1)}$ is a burrow in $Y^{(1)}$, then $A^\bullet(Y^{(1)}) \to A^\bullet(Z^{(1)})$ is surjective. Note that $A^\bullet(Y^{(1)})$ is generated as an algebra over $A^\bullet(Y)$ by the class of the exceptional divisor, and that $A^\bullet(Z^{(1)})$ is generated over $A^\bullet(Z)$ by the class of the exceptional divisor (if $Z^{(1)}$ is a blow-up of a burrow $Z$) or the hyperplane class (if $Z^{(1)}$ is a projective bundle over a burrow $Z$). In either case, surjectivity then follows from the fact that $A^\bullet(Y) \to A^\bullet(Z)$ is surjective, since the class of the exceptional divisor in $Y^{(1)}$ maps to the class of the exceptional divisor or hyperplane class, respectively. 
	
	We now argue that $[Z^{(1)}] \neq 0$ in $A^\bullet(Y^{(1)})$. If $Z^{(1)}$ is a blow-up of a burrow $Z$, then $[Z^{(1)}]$ maps to $[Z]$ under proper pushforward, so $[Z^{(1)}] \neq 0$. If $Z^{(1)}$ is the inverse image of a burrow $Z$ contained in $X$, then we may argue as follows. If $\pi \colon Y^{(1)} \to Y$ is the blow-up, then by the `key formula' \cite[Proposition 6.7(a)]{fultonintersectiontheory}, $\pi^\ast [Z] = j_\ast(c_{\mathrm{top}}(E))$, where $j \colon Z^{(1)} \to Y^{(1)}$ is the inclusion and $c_{\mathrm{top}}(E)$ is the top Chern class of the excess bundle. If $c$ denotes a class in $A^\bullet(Y^{(1)})$ whose restriction to $A^\bullet(Z^{(1)})$ is  $c_{\mathrm{top}}(E)$, then by the projection formula we have $\pi^\ast [Z] = j_\ast(1) \cdot c = [Z^{(1)}] \cdot c$. Since $\pi^\ast$ is injective we must have $[Z^{(1)}] \neq 0$. 
	The same argument works also if $Z^{(1)} = E_X \cap \widetilde Z$ for $Z$ not contained in $X$, using instead that $\pi^\ast [X \cap Z] = j_\ast (c_\mathrm{top}(Q))$ where $Q$ is the excess normal bundle of $\widetilde Z \to Z$. 

	So let us assume that $A^\bullet(Z)$ has Poincar\'e duality for all burrows $Z$. As every burrow in $Y^{(1)}$ is either a projective bundle over a burrow in $Y$ or a blow-up of one in a smaller burrow, and all classes $[Z]$ are nonzero, all burrows in $Y^{(1)}$ will have Poincar\'e duality by Propositions \ref{prop1} and \ref{prop2}. Since we have just verified that the hypothesis of the theorem are verified at each step of the construction, we are done by induction. \end{proof}

Propositions \ref{prop1}, \ref{prop2} and \ref{poincaredualitywonderful} remain valid with identical proof also for cohomology rings. We could also consider certain subalgebras of the Chow or cohomology rings, which is necessary for the applications to tautological classes that we will consider. Let us state this as a separate proposition:

\begin{prop}\label{poincaredualitywonderful2}Let $Y$ be a smooth variety. Let $G$ be a building set in $Y$. Suppose that for every burrow $Z \subseteq Y$  we have a subalgebra $R^\bullet(Z) \subseteq A^\bullet(Z)$ containing the Chern classes of all normal bundles, such that the collection $\{R^\bullet(Z)\}$ is closed under pullback and pushforwards. Define $R^\bullet(Y_G)$ as the algebra over $R^\bullet(Y)$ generated by all divisor classes $E_X$. Suppose that there exists an integer $d$ such that for every  burrow $Z$, $R^{d -\codim Z}(Z) \cong \Q$, and that $R^\bullet(Z)$ vanishes above this degree. Assume also that $[Z] \neq 0$ in $R^\bullet(Y)$ and that all restriction maps $R^\bullet(Y) \to R^\bullet(Z)$ are surjective. The following are equivalent: 
\begin{enumerate}[\qquad\scshape(1)]
\item $R^\bullet(Y_G)$ has Poincar\'e duality.
\item For every burrow $Z$, $R^\bullet(Z)$ has Poincar\'e duality.
\end{enumerate}
\end{prop}

\begin{proof}
The proof is identical to that of Proposition \ref{poincaredualitywonderful}. 
\end{proof}

\begin{rem} The algebra $R^\bullet(Y_G)$ could also have been described as the span of all elements of the form
	$$ \alpha \cdot E_{X_1} \cdots E_{X_k}$$
	with $\alpha \in R^\bullet(X_1 \cap \ldots \cap X_k)$. (It is not hard to see that this is well defined and that $R^\bullet(Y_G)$ is an algebra.) The equivalence of the two definitions is valid under the assumption that all restriction maps $R^\bullet(Y) \to R^\bullet(Z)$ are surjective. \end{rem}

\begin{thm} \label{maincor1} Fix $g \geq 2$ and $n \geq 1$. The following are equivalent:
\begin{enumerate}[\qquad\scshape(1)]
\item The tautological ring $R^\bullet(\M_{g,n}^{rt})$ has Poincar\'e duality.
\item The tautological rings $R^\bullet(\C_g^m)$ have Poincar\'e duality for all $m=1,\ldots, n$. 
\item The tautological ring $R^\bullet(\C_g^n)$ has Poincar\'e duality. 
\end{enumerate}
\end{thm} 

\begin{proof}Let us first dicuss the equivalence of (1) and (2).  Apply Proposition \ref{poincaredualitywonderful2} with $Y = \C_g^n$ and $G$ the set of diagonals. Then $Y_G \cong \M_{g,n}^{rt}$. Each burrow $Z$ is an intersection of diagonals, so it is isomorphic to some $\C_g^m$ and we can let $R^\bullet(Z)$ be its usual tautological ring. Then $R^\bullet(Y_G)$ (defined as in Proposition \ref{poincaredualitywonderful2}) coincides with the usual tautological ring $R^\bullet(\M_{g,n}^{rt})$. By \cite{looijengatautological,fabernonvanishing}, $R^{g-2+m}(\C_g^m) \cong \Q$, and the tautological ring vanishes above this degree. For each burrow $Z$, the restriction of the class $[Z]$ to any fiber of the map $\C_g^n \to \M_g$ is nonzero. The inclusion $\C_g^m \cong Z \hookrightarrow \C_g^n$ has a left inverse given by forgetting $(n-m)$ of the marked points, which implies that the restriction maps are all surjective. Thus the previous proposition applies, and we conclude that (1) and (2) are equivalent. 

We need to argue that (3) implies (2). So assume that $R^\bullet(\C_g^m)$ fails to have Poincar\'e duality for some $1 \leq m < n$. Then there exists $0 \neq \alpha \in R^\bullet(\C_g^m)$ which pairs to zero with everything in complementary degree. Let $\pi \colon \C_g^n \to \C_g^m$ be the forgetful map. By the projection formula, $\pi^\ast(\alpha)$ pairs to zero with everything in complementary degree, and the map $\pi^\ast$ is injective on Chow groups. Thus $R^\bullet(\C_g^n)$ fails to have Poincar\'e duality, too. \end{proof}

\begin{rem} \label{tavakolremark}Proposition \ref{poincaredualitywonderful2} generalizes and simplifies several results of Tavakol \cite[Section 5]{tavakol1}, \cite[Section 4]{tavakol2}, \cite[Section 8]{tavakolhyperelliptic}, \cite{tavakolconjecturalconnection} and \cite{tavakol0}. Roughly, Tavakol has in all cases proved statements of the following form: for certain algebraic curves $X$, the intersection matrix describing the pairing into the top degree in a tautological ring of a Fulton--MacPherson compactification $X[n]$ (or a similar space) is block triangular, with diagonal blocks expressed in terms of pairings in the tautological ring of $X^m$, $m\leq n$. Tavakol's proofs have used certain filtrations of the tautological rings, and explicit bases for the tautological rings given by `standard monomials'. 
\end{rem}

\subsection{$R^\bullet(\M_{g,n}^{rt})$ as an algebra over $R^\bullet(\C_g^n)$}

Suppose that $X \to S$ is a family of smooth varieties of relative dimension $m$ over a smooth base. Let $Y = X^n$ be the $n$th fibered power over $S$, and let $G$ be the building set given by the diagonal loci. In this case, the wonderful compactification $Y_G$ coincides with the Fulton--MacPherson compactification $X[n]$ introduced in \cite{fmcompactification}.  

Let $i \colon Z \hookrightarrow Y$ be a closed smooth subvariety of a smooth variety. We define a \emph{Chern polynomial} for $Z \subset Y$ to be a polynomial $P_{Z\subset Y}(t) \in A^\bullet(Y)[t]$ of the form
$$ t^d + c_1 t^{d-1} + \ldots + c_d $$
where $d = \codim Z$, $c_i$ is a class in $A^i(Y)$ whose restriction to $A^i(Z)$ is the $i$th Chern class of the normal bundle $N_{Z\subset Y}$, and $c_d = [Z]$. If the restriction map $A^\bullet(Y) \to A^\bullet(Z)$ is surjective then a Chern polynomial always exists, and in this case one can moreover simplify the blow-up formula in Equation \eqref{chowringdecomposition}: we have
$$ A^\bullet(\Bl_Z Y) = A^\bullet(Y)[E]/ \langle P_{Z\subset Y}(-E), J_{Z\subset Y} \cdot E\rangle, $$
where $J_{Z \subset Y}$ denotes the kernel of $A^\bullet(Y) \to A^\bullet(Z)$. 

For every $S \subseteq \{1,\ldots,n\}$, define 
$$ D_S = \{(x_1,\ldots,x_n) \in X^n \st x_i = x_j \text{ if } i,j \in S\}$$
and let $J_S = \ker(A^\bullet(X^n) \to A^\bullet(D_S))$. For any $i,j \in \{1,\ldots,n\}$ let $P_{i,j}(t) \in A^\bullet(X^n)[t]$ be the Chern polynomial of $D_{\{i,j\}} \hookrightarrow X^n$. The following presentation of $A^\bullet(X[n])$ as an algebra over $A^\bullet(X^n)$ is given in \cite[Theorem 6]{fmcompactification}, although their original proof contained a minor gap \cite{fmcompactificationfix}.

\begin{thm}[Fulton--MacPherson] \label{fmtheorem} There is an isomorphism $$A^\bullet(X[n]) \cong A^\bullet(X^n)[\{E_S\}]/\text{relations,}$$
in which there is a generator $E_S$ for every $S \subseteq \{1,\ldots,n\}$ with $\vert S \vert \geq 2$, and the relations are given by 
\begin{enumerate}[\qquad\scshape(1)]
\item $E_S \cdot E_T = 0$ unless $S \cap T \in \{\emptyset,S,T\}$,
\item $J_S \cdot E_S = 0$,
\item For any distinct $i,j \in \{1,\ldots,n\}$, $P_{i,j}(-\sum_{\{i,j\} \subseteq S} E_S) = 0$. 
\end{enumerate}
\end{thm}

\begin{rem} \label{altrel} One can give alternative presentations which are less economical but sometimes more practical. Suppose that $\{S_1,\ldots,S_k\}$ are disjoint subsets of $\{1,\ldots,n\}$ with $\vert S_i \vert \geq 2$, and that $S_i \subseteq T$ for $i=1,\ldots,k$. Let $W = \bigcap_{i=1}^k D_{S_i}$. If $k=0$ then we set $W=X^n$. Then there is a relation
$$ P_{D_T \subset W}(-\sum_{T \subseteq S}E_S) \cdot \prod_{i=1}^k E_{S_i} = 0.$$
This follows from the fact that $P_{D_T \subset W}(-\sum_{T \subseteq S}E_S)$ vanishes in $A^\bullet(\Bl_{D_T} W)$. The third class of relations in Theorem \ref{fmtheorem} is recovered when $k=0$ and $\vert T \vert = 2$.
\end{rem}

Theorem \ref{fmtheorem} also gives a presentation of $R^\bullet(\M_{g,n}^{rt})$ as an algebra over $R^\bullet(\C_g^n)$. To fix notation, let $D_S$ (as above) denote the diagonal loci in $\C_g^n$, and denote by the same symbol their classes in $A^\bullet(\C_g^n)$. Let $K \in A^1(\C_g)$ be the first Chern class of the relative dualizing sheaf of $\pi \colon \C_g \to \M_g$, and let $K_i \in A^1(\C_g^n)$ be the pullback of $K$ from the $i$th factor. Thus $R^\bullet(\C_g^n)$ is generated by the classes $D_{ij}$, $K_i$ and the $\kappa$-classes. (The $\kappa$-classes are pulled back from $\M_g$, where they are defined as $\kappa_d = \pi_\ast K^{d+1}$.) 

\begin{prop} \label{fmcor} There is an isomorphism $$R^\bullet(\M_{g,n}^{rt}) \cong R^\bullet(\C_g^n)[\{E_S\}]/\text{relations,}$$
in which there is a generator $E_S$ for every $S \subseteq \{1,\ldots,n\}$ with $\vert S \vert \geq 3$, and the relations are given by 
\begin{enumerate}[\qquad\scshape(1)]
\item $E_S \cdot E_T = 0$ unless $S \cap T \in \{\emptyset,S,T\}$,
\item $D_{ij} \cdot E_S = 0$ for $i \in S$, $j \notin S$,
\item $(D_{ij} + K_j) \cdot E_S = 0$ for $i, j \in S$,
\item For any distinct $i,j \in \{1,\ldots,n\}$, $\sum_{\{i,j\} \subseteq S} E_S = 0$. 
\end{enumerate}\end{prop}

\begin{proof}Theorem \ref{fmtheorem} gives a presentation of $A^\bullet(\M_{g,n}^{rt})$ over $A^\bullet(\C_g^n)$. The subalgebra generated over $R^\bullet(\C_g^n)$ by the exceptional divisors $E_S$ is exactly $R^\bullet(\M_{g,n}^{rt})$, and so we can read off a presentation for the tautological ring of $\M_{g,n}^{rt}$ from the theorem. 

Observe first that when $S = \{i,j\}$, we have $E_S = D_S$. Thus we can take the $E_S$ with $\vert S \vert \geq 3$ as generators. From the first relation in Theorem \ref{fmtheorem}, we see that we then need to impose the additional relation $D_{ij} \cdot E_S = 0$ for $i \in S$, $j \notin S$.

To determine the ideals $J_S$, observe firstly that the restriction map
$$ R^\bullet(\C_g^n) \to R^\bullet(D_S)$$
has a section. Indeed, the inclusion $D_S \hookrightarrow \C_g^n$ admits a left inverse given by forgetting all but one of the elements of $S$, say $s \in S$; the section is given by pullback along this left inverse. The image of this section is the algebra generated by the $\kappa$-classes and those $D_{ij}$ and $K_i$ with $i,j \notin S \setminus \{s\}$. 

Modulo the ideal $J_S$ we have the obvious relations $K_i - K_j =0 $ and $D_{ik} - D_{jk} =0$, where $i,j \in S$ and $k \notin S$. Moreover, the self-intersection formula implies that $D_{ij}^2 = -K_j D_{ij}$, so that $D_{ij}+K_j = 0$ modulo $J_S$ when $i,j \in S$. Modulo these relations every element of $R^\bullet(\C_g^n)$ is in the image of the section; that is, these relations can be used to reduce any polynomial in the $K_i$ and $D_{ij}$ to one in which none of the indices in $S \setminus \{s\}$ appear. Thus these relations generate $J_S$ and we can replace the relation $J_S \cdot E_S = 0$ with the third relation in our list, viz.\ $(D_{ij} + K_j) \cdot E_S = 0$ for $i, j \in S$. Notice that the relation $(D_{ik} - D_{jk}) \cdot E_S = 0$ for $i,j \in S$ and $k \notin S$ is not needed since it follows from the second relation in our list, and the relation $(K_i - K_j)\cdot E_S = 0$ is not needed since $(D_{ij} - K_j) - (D_{ij} - K_i) = (K_i-K_j)$. 

Finally, the Chern polynomial is given by $P_{ij}(t) = t + D_{ij}$. Keeping in mind that $E_{ij} = D_{ij}$, the final relation follows. \end{proof}

\begin{rem} A slightly less economical presentation can be obtained as in Remark \ref{altrel}. In this form, Proposition \ref{fmcor} was conjectured in \cite{tavakolconjecturalconnection}. \end{rem}

One can also give an additive description of $R^\bullet(\M_{g,n}^{rt})$ in terms of $R^\bullet(\C_g^m)$, $m \leq n$. More generally, for any wonderful compactification, the Chow groups of $Y_G$ were calculated in \cite{limotives}.

Let $G$ be a building set, and suppose that $N\subseteq G$ is a nest. A function $\mu \colon N \to \Z_{> 0}$ is called \emph{standard} if for all $X \in N$ we have
$$ \mu(X) < \codim(X) - \codim\left( \bigcap_{\substack{Z \in N \\ X \subsetneq Z}} Z\right). $$
For such a function, we denote $\Vert \mu \Vert = \sum_{X \in N} \mu(X)$. 

The following theorem specializes in particular to give a direct sum decomposition of $R^\bullet(\M_{g,n}^{rt})$ whose summands correspond to tautological rings of $\C_g^m$, $m \leq n$, with a degree shift. One can for instance express the Gorenstein discrepancies of $\M_{g,n}^{rt}$ in terms of those for $\C_g^m$ for $m \leq n$. We omit the details, as the procedure should be clear by now. 

\begin{thm}[Li] \label{lithm} Let $Y$ be a smooth variety, $G$ be a building set on $Y$. Then 
$$ A^\bullet(Y_G) = \bigoplus_{N} \bigoplus_\mu A^{\bullet - \Vert \mu \Vert} (\bigcap_{X \in N} X).  $$
Here the first summation runs over all nests $N \subseteq G$, and the second over all standard functions $\mu \colon N \to \Z_{>0}$. \end{thm}
Note that the summation includes in particular the {empty nest} $N = \emptyset$, corresponding to the single summand $A^\bullet(Y)$. 

To make sense of the ring structure on the right hand side of this isomorphism, let us define a map from the right hand side to the left hand side. Given a nest $N$, a standard function $\mu$ and an element $\alpha \in  A^{\bullet - \Vert \mu \Vert} (\bigcap_{X \in N} X)$, the element
$$ \alpha \cdot \prod_{X \in N} E_X^{\mu(X)}$$
is well defined in $A^\bullet(Y_G)$. To see that this map is surjective one uses relations analogous to those in Remark \ref{altrel}. Specifically, suppose that $N$ is a nest and $X \in N$. Let $Z_1,\ldots, Z_k$ be the minimal elements of $\{Z \in N \st X \subsetneq Z\}$, and let $W = \bigcap_{i=1}^k Z_i$. Then
\begin{equation}\tag{**}
P_{X\subset W}(-\sum_{S \subseteq X}E_S) \cdot \prod_{i=1}^k Z_i = 0. \label{relation}
\end{equation} 
In general a monomial $\prod_X E_X^{\mu(X)}$ can only be nonzero if the set $\{X \st \mu(X) > 0\}$ is a nest, and successive applications of the relation \eqref{relation} will reduce any such monomial to a linear combination of ones in which the exponents $\mu$ is a standard function. To see this, note that the degree of $P_{X\subset W}$ is 
$$ \codim X - \codim W = \codim(X) - \codim\left( \bigcap_{\substack{Z \in N \\ X \subsetneq Z}} Z\right), $$
which is precisely the upper bound appearing in the definition of a standard function. 

\begin{rem} The elements $ \alpha \cdot \prod_{X \in N} E_X^{\mu(X)}$ where $\mu$ is a standard function are equivalent to the \emph{standard monomials} used by Tavakol in several of his papers (see Remark \ref{tavakolremark}). Thus we see the connection between Tavakol's work and Li's Theorem \ref{lithm}.
\end{rem}
%

\begin{rem} Suppose as in Proposition \ref{poincaredualitywonderful} that there exists an integer $d$ such that for every  burrow $Z \subseteq Y$ (including $Z=Y$, corresponding to the empty nest), $A^{d -\codim Z}(Z) \cong \Q$ and $A^\bullet(Z)$ vanishes above this degree. Then it is not hard to prove that two summands $(N,\mu)$ and $(N',\mu')$ in the  decomposition can have a nontrivial pairing into the top degree only if $N=N'$ and 
$$ \mu(X) + \mu'(X) \geq \codim(X) - \codim\left( \bigcap_{\substack{Z \in N \\ X \subsetneq Z}} Z\right) $$
for all $X \in N$. It follows from this that the intersection matrices describing the pairing into the top degree are block-triangular. This gives another approach to Proposition \ref{poincaredualitywonderful}.  \end{rem}

\printbibliography

\end{document}